\documentclass[12pt]{amsart}
\newtheorem{theorem}{Theorem}[section]
\newtheorem{proposition}[theorem]{Proposition}
\newtheorem{lemma}[theorem]{Lemma}
\newtheorem{corollary}[theorem]{Corollary}
\newtheorem{remark}[theorem]{Remark}

\numberwithin{equation}{section}

\begin{document}

\title{Chen--Ruan cohomology of some moduli spaces, II}

\author[I. Biswas]{Indranil Biswas}

\address{School of Mathematics, Tata Institute of Fundamental
Research, Homi Bhabha Road, Mumbai 400005, India}

\email{indranil@math.tifr.res.in}

\author[M. Poddar]{Mainak Poddar}

\address{Theoretical Statistics and Mathematics Unit, Indian
Statistical Institute, 203 B. T. Road, Kolkata 700108, India}

\email{mainak@isical.ac.in}



\date{}

\begin{abstract}
Let $X$ be a compact connected Riemann surface of genus
at least two. Let $r$ be a prime number and $\xi\,
\longrightarrow\,X$ a holomorphic line bundle such that
$r$ is not a divisor of $\text{degree}(\xi)$. Let
${\mathcal M}_\xi(r)$ denote the moduli space of stable
vector bundles over $X$ of rank $r$ and determinant $\xi$.
By $\Gamma$ we will denote the group of line bundles $L$ over
$X$ such that $L^{\otimes r}$ is trivial. This group
$\Gamma$ acts on ${\mathcal M}_\xi(r)$ by the rule
$(E\, ,L)\,\longmapsto\, E\bigotimes L$. We compute the
Chen--Ruan cohomology of the corresponding orbifold.
\end{abstract}

\maketitle

\section{Introduction}\label{int}

Let $X$ be a compact connected Riemann surface of genus
$g$, with $g\,\geq\, 2$. Let ${\mathcal M}_\eta(2)$ denote
the moduli space of stable vector bundles $E$ over $X$
of rank two with $\bigwedge^2 E\,=\, \eta$, where $\eta\,
\longrightarrow\, X$ is a fixed holomorphic line bundle of
degree one. This ${\mathcal M}_\eta(2)$ is a smooth complex
projective variety of dimension $3g-3$.
Let $\Gamma_2\,=\, \text{Pic}^0(X)_2$ be the
group of holomorphic line bundles $L$ over $X$ such that
$L^{\otimes 2}$ is trivial. Such a line bundle $L$ defines
an involution of ${\mathcal M}_\eta(2)$ by sending any
$E$ to $E\bigotimes L$. This defines a holomorphic action of
$\Gamma_2$ on ${\mathcal M}_\eta(2)$.
In \cite{BP}, we computed the Chen--Ruan cohomology
algebra of the corresponding orbifold
${\mathcal M}_\eta(2)/\Gamma_2$.
(See \cite{CR1}, \cite{CR2}, \cite{Ru}, \cite{AR} for
Chen--Ruan cohomology.) We note that
${\mathcal M}_\eta(2)/\Gamma_2$ is the moduli space of
topologically nontrivial stable $\text{PSL}(2,
{\mathbb C})$--bundles over $X$.

Our aim here is to extend the above result to the moduli
spaces of vector bundles of rank $r$ over $X$, where $r$
is any prime number.

Fix a holomorphic line bundle $\xi\, \longrightarrow\, X$
of degree $d$ such that $d$ is not a multiple of the fixed
prime number $r$. Let ${\mathcal M}_\xi(r)$
denote the moduli space of stable vector bundles $E\,
\longrightarrow \,X$ of rank $r$ with $\bigwedge^r E\,=\, \xi$.
This moduli space ${\mathcal M}_\xi(r)$ is an
irreducible smooth complex
projective variety of dimension $(r^2-1)(g-1)$.
Let $\Gamma$ denote the
group of holomorphic line bundles $L$ over $X$ such that
$L^{\otimes r}$ is trivial. This group $\Gamma$ is isomorphic
to $({\mathbb Z}/r{\mathbb Z})^{\oplus 2g}$. Any line bundle
$L\,\in\,\Gamma$ defines a holomorphic automorphism of
${\mathcal M}_\xi(r)$ by sending any $E$ to $E\bigotimes L$.
These automorphisms together define a holomorphic action of
$\Gamma$ on ${\mathcal M}_\xi(r)$.

The corresponding orbifold ${\mathcal M}_\xi(r)/\Gamma$
is the moduli
space of stable $\text{PSL}(r,{\mathbb C})$--bundles
$F$ over $X$ such that the second Stiefel--Whitney class
$$
w_2(F)\,\in\, H^2(X,\, {\mathbb Z}/r{\mathbb Z})\,=\,
{\mathbb Z}/r{\mathbb Z}
$$
coincides with the image of $d$.

We compute the Chen--Ruan cohomology algebra of the
orbifold ${\mathcal M}_\xi(r)/\Gamma$.

We note that the results of Section \ref{sec2} and
Section \ref{sec3} are proved for all integers $r$. From
Section \ref{sec4} onwards we assume that $r$ is a prime
number.

\section{Group action and cohomology of fixed point
sets}\label{sec2}

Let $X$ be a compact connected Riemann surface of genus $g$,
with $g\,\geq\, 2$. Fix a holomorphic line bundle
\begin{equation}\label{e0}
\xi\, \longrightarrow\, X\, .
\end{equation}
Let $d\, \in\, {\mathbb Z}$ be the degree of $\xi$. Fix an integer
$r\, \geq\, 2$ which is coprime to $d$. Let ${\mathcal M}_\xi(r)$
denote the moduli space of stable vector bundles $E\,\longrightarrow
\, X$ of rank $r$ and $\det E\, :=\, \bigwedge^r E\,=\, \xi$. This
moduli space ${\mathcal M}_\xi(r)$ is an irreducible smooth complex
projective variety of dimension $(r^2-1)(g-1)$. The group of
all holomorphic automorphisms of ${\mathcal M}_\xi(r)$ will be
denoted by $\text{Aut}({\mathcal M}_\xi(r))$; it is known to be
a finite group.

Define
\begin{equation}\label{e1}
\Gamma\, :=\, \{L\,\in\, {\rm Pic}^0(X)\, \mid\,
L^{\otimes r}\,=\, {\mathcal O}_X\}\, .
\end{equation}
It is a group under the operation of tensor product of
line bundles. The order of $\Gamma$ is
$r^{2g}$. For any $L\, \in\, \Gamma$, let
\begin{equation}\label{e2}
\phi_L\,\in\, \text{Aut}({\mathcal M}_\xi(r))
\end{equation}
be the automorphism defined by $E\, \longmapsto\, E\bigotimes L$.
Let
\begin{equation}\label{e3}
\phi\, :\,\Gamma\, \longrightarrow\, \text{Aut}({\mathcal M}_\xi(r))
\end{equation}
be the homomorphism defined by $L\, \longmapsto\, \phi_L$. We
will describe the fixed point set
\begin{equation}\label{e03}
{\mathcal M}_\xi(r)^L\, :=\, {\mathcal M}_\xi(r)^{\phi_L}
\, \subset\, {\mathcal M}_\xi(r)
\end{equation}
of the automorphism $\phi_L$.

Take any nontrivial line bundle $L\, \in\, \Gamma\setminus
\{{\mathcal O}_X\}$. Let $\ell$ denote the order of
$L$. So $\ell$ is a divisor of $r$. Fix a nonzero
holomorphic section
$$
s\, :\, X\, \longrightarrow\, L^{\otimes \ell}\, .
$$
Define
\begin{equation}\label{e4}
Y_L\, :=\, \{z\in L\,\mid\, z^{\otimes\ell}\, \in\,
\text{image}(s)\}\, \subset\, L\, .
\end{equation}
Let
\begin{equation}\label{e5}
\gamma_L\, :\, Y_L\, \longrightarrow\, X
\end{equation}
be the restriction of the natural projection $L\, \longrightarrow
\,X$. Consider the action of the multiplicative group ${\mathbb
C}^*$ on the total space of $L$. The action of the subgroup
\begin{equation}\label{e04}
\mu_\ell\, :=\, \{c\, \in\, {\mathbb C}\, \mid\,c^\ell\, =\,1\}
\, \subset\, {\mathbb C}^*
\end{equation}
preserves the complex curve $Y_L$ defined in Eq. \eqref{e4}. In
fact, $Y_L$ is a principal $\mu_\ell$--bundle over $X$.
Since any two nonzero holomorphic sections of $L^{\otimes\ell}$
differ by multiplication with a nonzero constant
scalar, the isomorphism class of the principal $\mu_\ell$--bundle
$Y_L\,\longrightarrow\, X$ is independent of the choice of the
section $s$. In particular, the isomorphism class of
the complex curve $Y_L$ depends only on $L$.

Since the order of $L$ is exactly
$\ell$, the curve $Y_L$ is irreducible.

Let ${\mathcal N}^{Y_L}(d)$ denote the moduli space of
stable vector bundle bundles $F\, \longrightarrow\, Y_L$
of rank $r/\ell$ and degree $d$. We have a holomorphic submersion
\begin{equation}\label{e6}
\psi\, :\, {\mathcal N}^{Y_L}(d)\, \longrightarrow\,
\text{Pic}^d(X)
\end{equation}
that sends any $F$ to $\bigwedge^r \gamma_{L*}F$.

Let $\text{Gal}(\gamma_L)$ be the Galois group of
the covering $\gamma_L$ in Eq. \eqref{e5}; so
$\text{Gal}(\gamma_L)\,=\, \mu_\ell$.
For any $V\, \in\, {\mathcal N}^{Y_L}(d)$, clearly
$$
\rho^*V\, \in\, {\mathcal N}^{Y_L}(d)
$$
for all $\rho\, \in\, \text{Gal}(\gamma_L)$. Therefore, $\text{Gal}
(\gamma_L)$ acts on ${\mathcal N}^{Y_L}(d)$; the action of $\rho\,
\in\, \text{Gal}(\gamma_L)$ sends any $V$ to $\rho^*V$. We also
note that $\psi(V)\, =\, \psi(\rho^*V)$, where $\psi$ is the
projection in Eq. \eqref{e6}. Therefore, the action of $\text{Gal}
(\gamma_L)$ on ${\mathcal N}^{Y_L}(d)$ preserves the subvariety
\begin{equation}\label{e-2}
\psi^{-1}(\xi)\, \subset\, {\mathcal N}^{Y_L}(d)\, ,
\end{equation}
where $\xi$ is the line bundle in Eq. \eqref{e0}.

\begin{lemma}\label{lem1}
Take any nontrivial line bundle
$L\, \in\, \Gamma$. The fixed point set ${\mathcal
M}_\xi(r)^L$ (see Eq. \eqref{e03}) is identified with the quotient
variety $\psi^{-1}(\xi)/{\rm Gal}(\gamma_L)$ (see Eq.
\eqref{e-2}). The identification is defined by $F\,
\longmapsto\, \gamma_{L*}F$.
\end{lemma}

\begin{proof}
Let $0_X\subset\, L$ be the image of the zero section of $L
\,\longrightarrow\, X$.
Note that the pullback of $L$ to the
complement $L\setminus\{0_X\}$
has a tautological trivialization. Since $Y_L\,
\subset\, L\setminus\{0_X\}$, the holomorphic line bundle
$\gamma^*_L L$ is canonically trivialized. This trivialization
\begin{equation}\label{triv.}
\gamma^*_L L\, =\, {\mathcal O}_{Y_L}
\end{equation}
defines a holomorphic isomorphism
\begin{equation}\label{e-11}
V\,\longrightarrow\, V\bigotimes {\mathcal O}_{Y_L}\,=\,
V\bigotimes \gamma^*_L L 
\end{equation}
for any vector bundle $V\,\longrightarrow\, Y_L$; the
above isomorphism $V\,\longrightarrow\, V\bigotimes
{\mathcal O}_{Y_L}$ sends any $v$ to $v\bigotimes 1$.
Using projection formula, the isomorphism
in Eq. \eqref{e-11} gives an isomorphism
\begin{equation}\label{e-12}
\gamma_{L*} V\, \longrightarrow\, \gamma_{L*} (V\bigotimes
(\gamma^*_L L))\, =\, (\gamma_{L*} V) \bigotimes L\, .
\end{equation}

Note that
$$
\gamma^*_L\gamma_{L*} V\, =\, \bigoplus_{\rho\in
\text{Gal}(\gamma_L)} \rho^* V\, .
$$
Hence $\gamma^*_L\gamma_{L*} V$ is semistable if
$V$ is so. This implies that if $V$ is semistable, then the
vector bundle $\gamma_{L*} V$ is also semistable. Since $d
\,=\,\text{degree}(\xi)$ is coprime
to $r$, any semistable vector bundle over $X$ of rank $r$
and degree $d$ is stable. Therefore,
$\gamma_{L*} V\, \in\, {\mathcal M}_\xi(r)$
for each $V\,\in\, \psi^{-1}(\xi)$. In view
of the isomorphism in Eq. \eqref{e-12} we
conclude that
$$
\gamma_{L*} V\,\in\,{\mathcal M}_\xi(r)^L
$$
if $V\,\in\, \psi^{-1}(\xi)$.

Since $\gamma_{L*}F\, =\, \gamma_{L*}(\rho^* F)$, for all
$\rho\, \in\, \text{Gal}(\gamma_L)$, we get a morphism
\begin{equation}\label{whg}
\widehat{\gamma}\, :\,
\psi^{-1}(\xi)/\text{Gal}(\gamma_L)\, \longrightarrow\,
{\mathcal M}_\xi(r)^L
\end{equation}
defined by $V\, \longmapsto\, \gamma_{L*} V$.

To construct the inverse of the map $\widehat{\gamma}$, take
any $E\, \in\,{\mathcal M}_\xi(r)^L$. Fix an isomorphism
\begin{equation}\label{b0}
\theta'\, :\, E\,\longrightarrow\, E\bigotimes L\, .
\end{equation}
Since $E$ is stable, it follows that $E$ is simple; this
means that all automorphisms of $E$ are constant scalar
multiplications. Therefore, any two isomorphisms between
$E$ and $E\bigotimes L$ differ by multiplication
with a constant scalar. Let
\begin{equation}\label{b}
\theta\, \in\, H^0(X,\, End(E)\bigotimes L)
\end{equation}
be the section defined by $\theta'$ in Eq. \eqref{b0}.

Consider the pullback
$$
\gamma^*_L \theta\, \in\, H^0(Y_L,\, \gamma^*_L End(E)
\bigotimes\gamma^*_L L)
$$
of the section in Eq. \eqref{b}. Using the canonical trivialization
of the line bundle $\gamma^*_L L$ (see Eq. \eqref{triv.}), this
section $\gamma^*_L \theta$ defines a holomorphic section
\begin{equation}\label{b1}
\theta_0\, \in\, H^0(Y_L,\, \gamma^*_L End(E))\, =\,
H^0(Y_L,\, End(\gamma^*_L E))\, .
\end{equation}
Since $Y_L$ is irreducible (this was noted earlier) it
does not admit any nonconstant holomorphic functions, hence
the characteristic polynomial of $\theta_0(x)$ is independent of
the point $x\, \in\, Y_L$. Therefore, the set of eigenvalues of
$\theta_0(x)$ does not change as $x$ moves over $Y_L$.
Similarly, the multiplicity of each
eigenvalue of $\theta_0(x)$ is also independent of
$x\, \in\, Y_L$. Therefore, for each eigenvalue $\lambda$ of
$\theta_0(x)$, we have the associated generalized eigenbundle
\begin{equation}\label{eb}
\gamma^*_L E\, \supset\, E^\lambda\, \longrightarrow\, Y_L
\end{equation}
whose fiber over any $y\, \in\, Y_L$ is the
generalized eigenspace of $\theta_0(y)\, \in\, \text{End}
((\gamma^*_L E)_y)$ for the eigenvalue $\lambda$.

Recall that $Y_L$ was constructed by fixing a section $s$
of $L^{\otimes\ell}$ (see Eq. \eqref{e4}); it was also noted
that the isomorphism class of the covering $\gamma_L$ is independent
of the choice of $s$. We choose $s$ such that the
$\ell$--fold composition
\begin{equation}\label{c.c}
(\theta')^\ell\, :=\,
\overbrace{\theta'\circ\cdots\circ
\theta'}^{\ell\mbox{-}\rm{times}}\,:\,
E\, \longrightarrow\, E\bigotimes L^{\otimes\ell}
\end{equation}
coincides with $\text{Id}_E\bigotimes s$, where $\theta'$
is the homomorphism in Eq. \eqref{b0}. Since the vector
bundle $E$ is simple, there is exactly one
such section $s$. In fact,
$$
s\, =\, \text{trace}((\theta')^\ell)/r\, .
$$
We construct $Y_L$ using this $s$.

With this construction of $Y_L$, we have
$$
(\theta_0)^\ell\, :=\,
\overbrace{\theta_0\circ\cdots\circ\theta_0}^{
\ell\mbox{-}\rm{times}}\, =\, \text{Id}_{\gamma^*_L E}\, ,
$$
where $\theta_0$ is constructed in Eq. \eqref{b1}. Consequently,
the set of eigenvalues of $\theta_0(x)$ is contained
in $\mu_\ell$ (defined in Eq. \eqref{e04}); we noted earlier
that the set of eigenvalues of $\theta_0(x)$ along
with their multiplicities are independent of
$x\, \in\, Y_L$.

Since $Y_L$ is a principal $\mu_\ell$--bundle over $X$, the
Galois group $\text{Gal}(\gamma_L)$ is identified with
$\mu_\ell$. Note that $\text{Gal}(\gamma_L)$ has a natural
action on the vector bundle $\gamma^*_L E$ which is a lift of
the action
of $\text{Gal}(\gamma_L)$ on $Y_L$. Examining the construction
of $\theta_0$ (see Eq. \eqref{b1}) from $\theta'$, it follows
that the action of any
$$
\rho\, \in\, \text{Gal}(\gamma_L)\,=\, \mu_\ell
$$
on $\gamma^*_L E$ takes the eigenbundle
$E^\lambda$ (see Eq. \eqref{eb}) to the eigenbundle
$E^{\lambda\cdot\rho}$. This immediately implies that each
element of $\mu_\ell$ is an eigenvalue of $\theta_0(x)$,
and the multiplicities of the eigenvalues of $\theta_0(x)$
coincide. Hence, the multiplicity of each
eigenvalue of $\theta_0(x)$ is $r/\ell$.

Consider
\begin{equation}\label{E1}
E^1\, \longrightarrow\, Y_L\, ,
\end{equation}
which is the eigenbundle for the eigenvalue $1\, \in\, \mu_\ell$.
Define
$$
\widetilde{E}^1\, :=\, \bigoplus_{\rho\in \text{Gal}(\gamma_L)}
\rho^* E^1\, .
$$
There is a natural action of $\text{Gal}(\gamma_L)\,=\,
\mu_\ell$ on $\widetilde{E}^1$. Since the action of any $\rho
\, \in\,\mu_\ell$ on $\gamma^*_L E$ takes the eigenbundle
$E^\lambda$ to the eigenbundle
$E^{\lambda\cdot\rho}$, it follows immediately that
we have a $\text{Gal}(\gamma_L)$--equivariant identification
\begin{equation}\label{E2}
\gamma^*_L E\,=\, \widetilde{E}^1\, :=\,
\bigoplus_{\rho\in \text{Gal}(\gamma_L)}\rho^* E^1\, .
\end{equation}
In view of this $\text{Gal}(\gamma_L)$--equivariant
isomorphism we conclude that the composition
$$
E\, \longrightarrow\, \gamma_{L*}\gamma^*_L E
\,\stackrel{\sim}{\longrightarrow}\,
\gamma_{L*}\widetilde{E}^1\,=\,\gamma_{L*}
(\bigoplus_{\rho\in \text{Gal}(\gamma_L)}\rho^* E^1)
\,\longrightarrow\,\gamma_{L*}E^1
$$
is an isomorphism; here $$E\, \longrightarrow\,
\gamma_{L*}\gamma^*_L E$$ is the natural homomorphism
and $\bigoplus_{\rho\in \text{Gal}(\gamma_L)}\rho^* E^1
\,\longrightarrow\, E^1$ is the projection to the
direct summand corresponding to $\rho\,=\, 1$.

Since $E$ is stable, and $\gamma_{L*}E^1\,=\, E$, it follows
that $E^1$ is stable. Indeed, if a subbundle $F\,\subset\,
E^1$ violates the stability condition for $E^1$, then
the subbundle
$$
\gamma_{L*}F\,\subset\, \gamma_{L*}E^1\,=\, E
$$
violates the stability for $E$. Let
\begin{equation}\label{Phi}
\Phi\,:\, {\mathcal M}_\xi(r)^L\, \longrightarrow\,
\psi^{-1}(\xi)/\text{Gal}(\gamma_L)
\end{equation}
be the morphism that sends any $E$ to $E^1$. Since
$E$ is isomorphic to $\gamma_{L*}E^1$, it follows that
$$
\widehat{\gamma}\circ \Phi\, =\,
\text{Id}_{{\mathcal M}_\xi(r)^L}\, ,
$$
where $\widehat{\gamma}$ is constructed in Eq. \eqref{whg}.

Therefore, to complete the proof of the lemma it suffices
to show that for $F\, ,F'\, \in\, \psi^{-1}(\xi)$, if
\begin{equation}\label{P2}
\gamma_{L*}F\,=\,\gamma_{L*}F'\, ,
\end{equation}
then
\begin{equation}\label{c.}
F'\, =\, \tau^*F
\end{equation}
holds for some $\tau\, \in\,\text{Gal}(\gamma_L)$.

Take $F\, ,F'\, \in\, \psi^{-1}(\xi)$ such that Eq. \eqref{P2}
holds. Note that
\begin{equation}\label{c1.}
\bigoplus_{\tau,\eta\in \text{Gal}(\gamma_L)}
Hom(\eta^*F\, , \tau^* F')\, =\, \bigoplus_{\tau\in
\text{Gal}(\gamma_L)}Hom(F\, , \tau^*F')^{\oplus\ell}\,=\,
Hom(\gamma^*_L\gamma_{L*}F\, ,\gamma^*_L\gamma_{L*}F')\, .
\end{equation}
Since $\gamma_{L*}F\,=\,\gamma_{L*}F'$,
$$
H^0(X, \, Hom(\gamma_{L*}F\, ,
\gamma_{L*}F'))\, \not=\, 0\, .
$$
Consequently,
$$
H^0(Y_L, \, Hom(\gamma^*_L\gamma_{L*}F\, ,
\gamma^*_L\gamma_{L*}F'))\, \not=\, 0\, .
$$
Therefore, from Eq. \eqref{c1.} we conclude that there is
some $\tau\, \in\,\text{Gal}(\gamma_L)$ such that
\begin{equation}\label{c2.}
H^0(Y_L, \, Hom(F\, , \tau^*F'))\, \not=\, 0\, .
\end{equation}
Since both $F$ and $\tau^*F'$ are stable vector bundles
with
$$
\frac{\text{degree}(F)}{\text{rank}(F)}\, =\,
\frac{\text{degree}(\tau^*F')}{\text{rank}(\tau^*F')}\, ,
$$
{}from Eq. \eqref{c2.} we conclude that the vector bundle $F$ is
isomorphic to $\tau^*F'$. In other words, Eq. \eqref{c.} holds.
This completes the proof of the lemma.
\end{proof}

In Lemma \ref{lem-3}  we will show that the action of
${\rm Gal}(\gamma_L)$ on $\psi^{-1}(\xi)$ is free.

In the next section we will investigate the action of
the isotropy subgroups on the tangent bundle for the
action of $\Gamma$ of ${\mathcal M}_\xi(r)$.

\section{Action on the tangent bundle}\label{sec3}

Fix any $L\, \in\, \Gamma \setminus\,\{{\mathcal O}_X\}$.
Take any
$$
E\, \in\, {\mathcal M}_\xi(r)^L
$$
(see Eq. \eqref{e03}). Fix an isomorphism $\theta'$
as in Eq. \eqref{b0}. Let $s$ be the unique section of
$L^{\otimes\ell}$ such that the homomorphism
$(\theta')^\ell$ in Eq. \eqref{c.c} coincides with
$\text{Id}_E\bigotimes s$. As we noted earlier,
$s\, =\, \text{trace}((\theta')^\ell)/r$. Let
$$
\gamma_L\, :\, Y_L\, \longrightarrow\, X
$$
be the covering $Y_L$ constructed as in Eq. \eqref{e4} using
this section $s$. Recall that $\text{Gal}(\gamma_L)\,=\,
\mu_\ell$, where $\ell$ is the order of $L$.

Let
\begin{equation}\label{F}
F\, :=\, E^1\, \longrightarrow\, Y_L
\end{equation}
be the vector bundle constructed in Eq. \eqref{E1}.
The decomposition of $\gamma^*_L E$ in Eq. \eqref{E2} yields
the following decomposition of the pullback
$\gamma^*_L End(E)\,=\,End(\gamma^*_L E)$:
\begin{equation}\label{e9}
\gamma^*_L End(E)\, =\, \bigoplus_{t\in\text{Gal}(\gamma_L)}
\bigoplus_{u\in\text{Gal}(\gamma_L)}Hom(u^* F\, ,
(ut)^*F)\, =\,
\bigoplus_{t\in\text{Gal}(\gamma_L)}
\bigoplus_{u\in\text{Gal}(\gamma_L)} (u^* F)^\vee\bigotimes
(ut)^*F
\end{equation}
(see also Eq. \eqref{c1.}).

Note that for each $t\,\in\,\text{Gal}(\gamma_L)$, the
vector bundle
\begin{equation}\label{e10}
{\mathcal E}_t\, :=\, \bigoplus_{u\in\text{Gal}(\gamma_L)}
Hom(u^* F\, , (ut)^*F)\,\longrightarrow\, Y_L
\end{equation}
in Eq. \eqref{e9} is left invariant by the
natural action of $\text{Gal}
(\gamma_L)$ on the vector bundle $\gamma^*_L End(E)$. Therefore,
${\mathcal E}_t$ descends to $X$. Let
\begin{equation}\label{e11}
{\mathcal F}_t\, \longrightarrow\, X
\end{equation}
be the descent of ${\mathcal E}_t$. So
\begin{equation}\label{e1a}
{\mathcal F}_t\, =\, \gamma_{L*}Hom(F\, , t^* F)\, ,
\end{equation}
and
$$
\gamma^*_L {\mathcal F}_t\,=\, {\mathcal E}_t\, .
$$

The decomposition
$$
\gamma^*_L End(E)\, =\, \bigoplus_{t\in\text{Gal}(\gamma_L)}
{\mathcal E}_t
$$
in Eq. \eqref{e9} is preserved by the action of
$\text{Gal}(\gamma_L)$. Therefore, this decomposition
descends to the following decomposition:
\begin{equation}\label{e12}
End(E)\, =\, \bigoplus_{t\in\text{Gal}(\gamma_L)}
{\mathcal F}_t\, .
\end{equation}

We will now describe the differential
$d\phi_L(E)$, where $\phi_L$ is the automorphism
in Eq. \eqref{e2}, and $E\,\in\,{\mathcal M}_\xi(r)^L$.

Recall that we fixed an isomorphism
$\theta'\, :\, E\,\longrightarrow\, E\bigotimes L$
as in Eq. \eqref{b0}. This isomorphism $\theta'$
induces an isomorphism of the endomorphism bundle
$End(E)$ with $End(E\bigotimes L)$. Since
$$
End(E)\,=\, E\bigotimes E^\vee\,=\,
(E\bigotimes L)\bigotimes (E\bigotimes L)^\vee
\,=\, End(E\bigotimes L)\, ,
$$
the isomorphism of $End(E)$ with $End(E\bigotimes L)$
defined by $\theta'$ gives an automorphism of $End(E)$. Let
\begin{equation}\label{wth}
\widehat{\theta}\, :\, End(E)\,\longrightarrow\,
End(E)
\end{equation}
be this automorphism constructed from $\theta'$.
Since any two isomorphisms between $E$ and $E\bigotimes
L$ differ by a constant scalar, the automorphism
$\widehat{\theta}$ is independent of the choice of
$\theta'$.

Let
$$
{\rm ad}(E)\, \subset\, End(E)
$$
be the holomorphic subbundle of corank one given
by the sheaf of trace zero endomorphisms. Clearly,
$$
\widehat{\theta}({\rm ad}(E))\, \subset\,
{\rm ad}(E)\, .
$$
We note that $T_E{\mathcal M}_\xi(r)\,=\,
H^1(X,\, \text{ad}(E))$; here $T$ denotes the
holomorphic tangent bundle. Let
\begin{equation}\label{ot}
\overline{\theta}\, :\, H^1(X,\, \text{ad}(E))
\,\longrightarrow\,H^1(X,\, \text{ad}(E))
\end{equation}
be the automorphism induced by $\widehat{\theta}$.

{}From the construction of $\phi_L$ it follows that
the differential
\begin{equation}\label{dphi}
d\phi_L(E)\, :\, T_E{\mathcal M}_\xi(r)\,
\longrightarrow\,T_E{\mathcal M}_\xi(r)
\end{equation}
coincides with $\overline{\theta}$ constructed in Eq. \eqref{ot}.

In the proof of Lemma \ref{lem1} we observed
that the pullback to $Y_L$ of the isomorphism $\theta'$
coincides with the isomorphism
$$
\gamma^*_L E\, \longrightarrow\, \gamma^*_L E\bigotimes
\gamma^*_L  L
$$
obtained by tensoring with the tautological section of
$\gamma^*_L L$ (see Eq. \eqref{triv.}). Consider the
automorphism $\widehat{\theta}$ in Eq. \eqref{wth} induced
by $\theta'$. From the above description of $\theta'$
it follows immediately that $\widehat{\theta}$ acts
on the subbundle ${\mathcal E}_t$ (see Eq.
\eqref{e10}) as multiplication by
$$
t\, \in\, \mu_\ell\, \subset\, {\mathbb C}^*\, .
$$

\begin{lemma}\label{inc1}
Take any
\begin{equation}\label{t}
t\, \in\,\mu_\ell\setminus \{1\}\, .
\end{equation}
($\ell$ is the order of $L$; see Eq. \eqref{e04} for
$\mu_\ell$). Consider ${\mathcal F}_t$ constructed in Eq.
\eqref{e11} (see also Eq. \eqref{e12}). Then
\begin{equation}\label{in}
{\mathcal F}_t\, \subset\, {\rm ad}(E)\, .
\end{equation}
\end{lemma}

\begin{proof}
We first note that
$\dim H^0(X,\, End(E))\,=\,1$ because
$E$ is stable. On the other hand,
$$
\dim H^0(X,\,\gamma_{L*}Hom(F\, ,{\rm Id}^* F))\, \geq\, 1\, ,
$$
where ${\rm Id}\, =\, 1\, \in\, \text{Gal}(\gamma_L)\,=\,
\mu_\ell$ is the identity element. Therefore, from
Eq. \eqref{e1a} and Eq. \eqref{e12},
\begin{equation}\label{in1}
H^0(X,\, {\mathcal F}_t)\, =\, 0\, ,
\end{equation}
where $t$ is the element in Eq. \eqref{t}.

\begin{remark}\label{rem-st}
{\rm Since the vector bundle $E$ is stable, it admits a
unique Hermitian--Einstein connection. The connection on
$End(E)$ induced by a Hermitian--Einstein connection on $E$ is
also Hermitian--Einstein. Therefore, the vector bundle
$End(E)$ is polystable of degree zero.}$\hfill{\Box}$
\end{remark}

Continuing with the proof of Lemma \ref{inc1},
since $End(E)$ is polystable of degree zero, and ${\mathcal F}_t$
is a direct summand of $End(E)$ (see Eq. \eqref{e1a} and
Eq. \eqref{e12}), it follows that ${\mathcal F}_t$ is also
polystable of degree zero. Consider the trace map
$$
End(E)\, \supset\, {\mathcal F}_t\, \stackrel{\rm
trace}{\longrightarrow}\, {\mathcal O}_X\, .
$$
Since ${\mathcal F}_t$ is polystable of degree zero, if
the above trace homomorphism on ${\mathcal F}_t$ is nonzero, then
${\mathcal O}_X$ is a direct summand of ${\mathcal F}_t$.
If ${\mathcal O}_X$ is a direct summand of ${\mathcal F}_t$, then
Eq. \eqref{in1} is contradicted. Hence the trace map on
${\mathcal F}_t$ vanishes identically. This implies that Eq.
\eqref{in} holds. This completes the proof of the lemma.
\end{proof}

Consider the tautological trivialization of the line
bundle $\gamma^*_L L$ (see Eq. \eqref{triv.}). The action of
any element $t\, \in\,
\text{Gal}(\gamma_L)\, =\, \mu_\ell$ takes the
tautological section of $\gamma^*_L L$ to $t^{-1}$--times
the tautological section. Using this it follows immediately
that $t$ acts on ${\mathcal E}_t$ in Eq. \eqref{e10} as
multiplication by $t$. Now from the construction of
the automorphism $\widehat{\theta}$ in Eq. \eqref{wth} it
follows that $\widehat{\theta}$ acts on
${\mathcal F}_t$ as multiplication by $t$. In view of
this and Eq. \eqref{in}, we conclude that for all
$t\,\in\, \mu_\ell\setminus\{1\}$,
the automorphism $\overline{\theta}$
in Eq. \eqref{ot} acts on the subspace
\begin{equation}\label{su}
H^1(X,\, {\mathcal F}_t)\, \subset\, H^1(X,\,
\text{ad}(E))\,=\, T_E{\mathcal M}_\xi(r)
\end{equation}
as multiplication by $t$.

We will calculate the dimension of the subspace
$H^1(X,\, {\mathcal F}_t)$ in Eq. \eqref{su},
where $t\,\in\, \mu_\ell\setminus\{1\}$. Note that
$$
\text{rank}({\mathcal F}_t)\,=\, r^2/\ell\,\,~~\,\,
\text{~and~}\,\,~~\,\,\text{degree}({\mathcal F}_t)
\,=\, 0\, .
$$
Since $H^0(X,\, \text{ad}(E))\, =\, 0$, from
Eq. \eqref{in} we have
\begin{equation}\label{z}
H^0(X,\, {\mathcal F}_t)\, =\, 0\, .
\end{equation}
Therefore, the Riemann--Roch theorem says that
$$
\dim H^1(X,\, {\mathcal F}_t)\, =\,
r^2(g-1)/\ell\, .
$$

Hence, we have proved the following lemma:

\begin{lemma}\label{lem-a}
Take any $t\, \in\, \mu_\ell\setminus \{1\}$. Then
${\mathcal F}_t\, \subset\, {\rm ad}(E)$,
and the automorphism $\overline{\theta}$ in Eq.
\eqref{ot} acts on the subspace
$$
H^1(X,\, {\mathcal F}_t)\, \subset\, H^1(X,\,
{\rm ad}(E))\,=\, T_E{\mathcal M}_\xi(r)
$$
as multiplication by $t$. Also,
$$
\dim H^1(X,\, {\mathcal F}_t)\,= \,
r^2(g-1)/\ell\, .
$$
\end{lemma}

Now we set $t\,=\,1\,\in\, \mu_\ell\,=\, \text{Gal}(\gamma_L)$. The
automorphism $\widehat{\theta}$ of $End(E)$ acts trivially
on the subbundle ${\mathcal F}_1\,\subset \, End(E)$. Therefore,
$\overline{\theta}$ acts trivially on the subspace
$$
H^1(X,\, {\mathcal F}_1)\bigcap H^1(X,\, \text{ad}(E))
\, \subset\, H^1(X,\, \text{ad}(E))\,=\,
T_E{\mathcal M}_\xi(r)\, .
$$
{}From Eq. \eqref{e12} and Eq. \eqref{in},
$$
\text{ad}(E)\, =\, ({\mathcal F}_1\bigcap \text{ad}(E))
\bigoplus (\bigoplus_{\tau\in\mu_\ell\setminus\{1\}}
{\mathcal F}_\tau)\, .
$$
Consequently, the dimension of the subspace of
$H^1(X,\, \text{ad}(E))$ on which $\overline{\theta}$ in
Eq. \eqref{ot} acts as the identity map is
$$
\dim H^1(X,\, \text{ad}(E))
-\sum_{t\in\mu_\ell\setminus \{1\}}\dim H^1(X,\, {\mathcal F}_t) $$
$$
=\,(r^2-1)(g-1)- \frac{(\ell-1)r^2(g-1)}{\ell}\,=\,
\frac{r^2(g-1)}{\ell}-g+1\, .
$$

Combining this with Lemma \ref{lem-a} we get the
following proposition.

\begin{proposition}\label{prop1}
The eigenvalues of the differential
$$
d\phi_L(E)\, :\, T_E{\mathcal M}_\xi(r)\,
\longrightarrow\,T_E{\mathcal M}_\xi(r)
$$
are $\mu_\ell$. For any $t\, \in\,
\mu_\ell\setminus\{1\}$, the multiplicity
of the eigenvalue $t$ is $r^2(g-1)/\ell$.
The multiplicity of the eigenvalue $1$ is
$1-g+r^2(g-1)/\ell$.
\end{proposition}

\begin{lemma}\label{lem-3}
The action of ${\rm Gal}(\gamma_L)$ on
$\psi^{-1}(\xi)$ in Lemma \ref{lem1} is free.
\end{lemma}

\begin{proof}
In Eq. \eqref{z} we saw that
$H^0(X,\, {\mathcal F}_t)\, =\, 0$
for all $t\,\in\,\text{Gal}(\gamma_L)\setminus\{1\}$.
Now, from Eq. \eqref{e1a} it follows that
$$
H^0(X,\, {\mathcal F}_t)\,=\,
H^0(Y_L,\, Hom(F\, , t^* F))\,=\,0
$$
for all $t\,\in\,\text{Gal}(\gamma_L)\setminus\{1\}$.
Since $F\,=\, E^1$ (see Eq. \eqref{F}), this implies
that for each $t\,\in\,\text{Gal}(\gamma_L)\setminus
\{1\}$, the vector bundle $E^1$ is not isomorphic to
$t^*E^1$. This completes the proof of the lemma.
\end{proof}

\section{Intersection of fixed point sets}\label{sec4}

Henceforth, we will always assume that $r$ is a prime
number. As before, $d\,=\,\text{degree}(\xi)$ is
assumed to be coprime to $r$.

We note that the group $\Gamma$
(see Eq. \eqref{e1}) is a vector space over the
field ${\mathbb Z}/r{\mathbb Z}$. For any $J\, ,L\, \in\,
\Gamma$, and any $n\, \in\, {\mathbb Z}/r{\mathbb Z}$,
$$
J+L\, :=\, J\bigotimes L~\,~\,~\, \text{~and~}
~\,~\,~\,  nL\, :=\, L^{\otimes n}\, .
$$
For a line bundle $L_0$ over $X$, by $L^{\otimes 0}_0$ we
will denote the trivial line bundle ${\mathcal O}_X$.

Take two linearly independent elements
\begin{equation}\label{e13}
J\, , L\, \in\, \Gamma\, .
\end{equation}
We note that the line bundles $J^{\otimes i}
\bigotimes L^{\otimes j}$, $i\, ,j\, \in\, [1\, ,r-1]$,
are all distinct and nontrivial.

Take any
$$
E\, \in\, {\mathcal M}_\xi(r)^J\bigcap {\mathcal M}_\xi(r)^L
$$
(see Eq. \eqref{e03}), where $J$ and $L$ are, as in
Eq. \eqref{e13}, linearly independent. Fix isomorphisms
\begin{equation}\label{e14}
\theta_1\, :\, E\, \longrightarrow\, E\bigotimes J
~\,\,~ \text{~and~}~\,\,~ \theta_2\, :\, E\, \longrightarrow\,
E\bigotimes L\, .
\end{equation}
The isomorphism $\theta_1$ (respectively, $\theta_2$) gives
an inclusion of the line bundle $J^\vee\,=\,
J^{\otimes (r-1)}$ (respectively, $L^\vee\,=\,
L^{\otimes (r-1)}$) in
$End(E)\, =\, E\bigotimes E^\vee$. Note that since
$r$ is a prime number, any $J^{\otimes i}$ (respectively,
$L^{\otimes i}$) is a tensor power of $J^\vee$
(respectively, $L^\vee$). Using the associative algebra
structure of the fibers of $End(E)$ defined by
composition of homomorphisms, the above homomorphisms
$$
J^\vee\, \longrightarrow\, End(E)\,~\,~\, \text{~and~}
L^\vee\, \longrightarrow\, End(E)
$$
give an inclusion of $(J^\vee)^{\otimes i}\bigotimes
(L^\vee)^{\otimes j}$ in $End(E)$ for all $i\, , j\, \geq\,
0$. Therefore, we get an inclusion of $J^{\otimes
i}\bigotimes
L^{\otimes j}$ in $End(E)$ for all $i\, , j\, \geq\, 0$.
The line bundle
$J^{\otimes 0}\bigotimes L^{\otimes 0}\, =\, {\mathcal O}_X$
sits inside $End(E)$ as scalar multiplications.

Let
\begin{equation}\label{e15}
\Theta\, :\, \bigoplus_{i,j=0}^{r-1} J^{\otimes i}\bigotimes
L^{\otimes j}\, \longrightarrow\, End(E)
\end{equation}
be the homomorphism constructed as above.

\begin{lemma}\label{th.is.}
The homomorphism $\Theta$ in Eq. \eqref{e15} is a holomorphic
isomorphism of vector bundles.
\end{lemma}

\begin{proof}
Take any proper subset
$$
S\, \subset\, [0\, ,r-1]\times [0\, ,r-1]\, ,
$$
and also take any $(i_0\, ,j_0)\,\in\, [0\, ,r-1]\times
[0\, ,r-1]\setminus S$ in the complement. Assume that the
restriction of the homomorphism $\Theta$ in Eq. \eqref{e15} to
$$
W_S\, :=\, \bigoplus_{(i,j)\in S} J^{\otimes i}\bigotimes
L^{\otimes j}
$$
is injective.
Note that $W_S$ is polystable of degree zero, and
in Remark \ref{rem-st} we observed that the vector bundle $End(E)$
is polystable.
Since both $W_S$ and $End(E)$ are polystable vector bundles of
degree zero, it follows that $\Theta(W_S)$ is a subbundle of
$End(E)$. Furthermore, there is a holomorphic subbundle
$$
W'_S\, \subset\, End(E)
$$
which is a direct summand of $\Theta(W_S)$. In particular,
\begin{equation}\label{ds}
End(E)\,=\, \Theta(W_S)\bigoplus W'_S\, .
\end{equation}
We have
$$
H^0(X,\, Hom(J^{\otimes i_0}\bigotimes
L^{\otimes j_0}\, , W_S))\, =\, 0
$$
because $J^{\otimes i_0}\bigotimes
L^{\otimes j_0}$ is distinct from
$J^{\otimes i}\bigotimes L^{\otimes j}$
for all $(i\, ,j)\,\in\, S$. Therefore, the
projection of $J^{\otimes i_0}\bigotimes
L^{\otimes j_0}$ to the direct summand
$\Theta(W_S)\, \subset\, End(E)$ in Eq. \eqref{ds}
vanishes identically. This implies that
$$
\Theta(J^{\otimes i_0}\bigotimes L^{\otimes j_0})\,
\subset\, W'_S\, .
$$
Hence $\Theta$ makes $W_S\bigoplus (J^{\otimes i_0}\bigotimes
L^{\otimes j_0})$ a subbundle of $End(E)$.

Now, using
induction we conclude that $\Theta$
in Eq. \eqref{e15} is a pointwise injective
homomorphism of vector bundles. Since
$$
r^2\, =\, \text{rank}(End(E))\, =\,
\text{rank}(\bigoplus_{i,j=0}^{r-1} J^{\otimes i}\bigotimes
L^{\otimes j})\, ,
$$
it follows that $\Theta$ is an isomorphism of vector bundles.
\end{proof}

Take two vector bundles $E\, ,F\,\in\, {\mathcal M}_\xi(r)$. If
the vector bundle $End(E)$ is isomorphic to $End(F)$, then
at least one of the following two statements is valid:
\begin{enumerate}
\item There is a line bundle $\zeta\, \longrightarrow\, X$ such
that
\begin{equation}\label{det0}
E\, =\, F\bigotimes\zeta\, .
\end{equation}

\item There is a line bundle $\zeta\, \longrightarrow\,
X$ such that
\begin{equation}\label{det00}
E\, =\, F^\vee\bigotimes\zeta\, .
\end{equation}
\end{enumerate}

Since $\bigwedge^r E\, =\,\bigwedge^r F$, taking the $r$--th
exterior power of both sides of Eq. \eqref{det0} it follows
that $\zeta^{\otimes r}\,=\, {\mathcal O}_X$.

Similarly, since $\bigwedge^r E\, =\,\bigwedge^r F\,=\,
\xi$, taking the $r$--th
exterior power of both sides of Eq. \eqref{det00} it follows
that $\zeta^{\otimes r}\,=\, \xi^{\otimes 2}$. Recall that
$r$ is a prime number, and $\text{degree}(\xi)$ is not
a multiple of $r$. Hence, if Eq. \eqref{det00} holds, then
$r\, =\,2$.

Therefore, Lemma \ref{th.is.} has the following corollary:

\begin{corollary}\label{cor1}
Take two linearly independent elements
$J\, , L\, \in\, \Gamma$. If
$$
E\, , F\, \in\, {\mathcal M}_\xi(r)^J\bigcap {\mathcal M}_\xi
(r)^L\, ,
$$
then at least one of the following two statements is valid:
\begin{enumerate}
\item There is a line
bundle $\zeta\, \in\,\Gamma$ such
that $F\bigotimes\zeta$ is holomorphically isomorphic to $E$.

\item There is a line
bundle $\zeta\, \in\,\Gamma$ such
that $F^\vee\bigotimes\zeta$ is holomorphically isomorphic
to $E$.
\end{enumerate}
If the second statement holds, then $r\,=\, 2$.
\end{corollary}

\begin{remark}\label{rem1}
{\rm Take a line bundle $L\,\in\,\Gamma$, and also take
a vector bundle $E\,\in\, {\mathcal M}_\xi(r)^L$. Let
$$
f_0\, :\, E\, \longrightarrow\, E\bigotimes L
$$
be a holomorphic isomorphism.}
\begin{itemize}
\item {\rm Take any $\zeta\, \in\, \Gamma$.
Then
$$
f_0\otimes {\rm Id}_\zeta\, :\,
E\bigotimes \zeta\, \longrightarrow\, E\bigotimes L\bigotimes\zeta
\,=\, E\bigotimes \zeta\bigotimes L
$$
is also an isomorphism. Hence $E\bigotimes \zeta\, \in\,
{\mathcal M}_\xi (r)^L$.
Therefore, if
$$
E\, \, \in\, {\mathcal M}_\xi(r)^J\bigcap {\mathcal M}_\xi
(r)^L\, ,
$$
and $\zeta\, \in\, \Gamma$, then
$$
E\bigotimes \zeta\, \, \in\, {\mathcal M}_\xi(r)^J\bigcap
{\mathcal M}_\xi (r)^L\, .
$$}

\item {\rm Consider the isomorphism
\begin{equation}\label{1e}
f^\vee_0\, :\,
E^\vee\bigotimes L^\vee\, =\, (E\bigotimes L)^\vee\,
\longrightarrow\, E^\vee\, .
\end{equation}
For any holomorphic line bundle
$\eta\, \longrightarrow\, X$, tensoring the isomorphism
in Eq. \eqref{1e} by
$$
\text{Id}_{L\otimes\eta}\, :\, L\bigotimes\eta\,
\longrightarrow\,L\bigotimes\eta
$$
we get an isomorphism
$$
E^\vee\bigotimes\eta\, \longrightarrow\,
E^\vee\bigotimes L \bigotimes \eta\,=\,
E^\vee\bigotimes\eta \bigotimes L\, .
$$
Therefore, if $\bigwedge^r (E^\vee\bigotimes\eta)\,=\,
\xi$, then $E^\vee\bigotimes\eta\, \in\,
{\mathcal M}_\xi (r)^L$. Now using the first part
of the remark we conclude the following: if
$$
E\, \, \in\, {\mathcal M}_\xi(r)^J\bigcap {\mathcal M}_\xi
(r)^L\, ,
$$
and $\bigwedge^r (E^\vee\bigotimes\eta)\,=\,
\xi$, then
$$
E^\vee\bigotimes\eta\bigotimes \zeta\, \, \in\, {\mathcal
M}_\xi(r)^J\bigcap
{\mathcal M}_\xi (r)^L
$$
for all $\zeta\, \in\, \Gamma$.}
\end{itemize}$\hfill{\Box}$
\end{remark}

\begin{remark}\label{rem-1}
{\rm Assume that $r\,=\,2$. Take any
$E\,\in\,{\mathcal M}_\xi(2)$. Since
$\bigwedge^2 E\,=\, \xi$, contracting
both sides by $E^\vee$ it follows that 
$E\,=\, E^\vee\bigotimes\xi$.}$\hfill{\Box}$
\end{remark}

As before, take two linearly independent elements
$J\, , L\, \in\, \Gamma$. Consider the coverings
$$
\gamma_J\, :\, Y_J\, \longrightarrow\, X ~\,~\, \text{~and~}
~\,~\, \gamma_L\, :\, Y_L\, \longrightarrow\, X
$$
constructed
as in Eq. \eqref{e5} from $J$ and $L$ respectively. Since the
Galois groups of both $\gamma_J$ and $\gamma_L$ are
${\mathbb Z}/r{\mathbb Z}$, we get surjective homomorphisms
\begin{equation}\label{rl1}
\rho_J\, :\, H_1(X,\, {\mathbb Z})\, \longrightarrow\,
{\mathbb Z}/r{\mathbb Z}
\end{equation}
and
\begin{equation}\label{rl2}
\rho_L\, :\, H_1(X,\, {\mathbb Z})\, \longrightarrow\,
{\mathbb Z}/r{\mathbb Z}\, .
\end{equation}
Using $\rho_J$ and $\rho_L$, we will construct a
homomorphism from
$H_1(X,\, {\mathbb Z})$ to $\text{PGL}(r,{\mathbb C})$.

Let $D$ be the $r\times r$ diagonal matrix whose
$(j\, ,j)$--th
entry is $\exp(2\pi\sqrt{-1}j/r)$. Let
\begin{equation}\label{rl3}
\rho'_1\, :\, {\mathbb Z}/r{\mathbb Z}\, \longrightarrow\,
\text{PGL}(r,{\mathbb C})
\end{equation}
be the homomorphism that sends any $\tau$ to the image of
$D^\tau$ in $\text{PGL}(r,{\mathbb C})$.

Let
$$
R\, \in\, \text{GL}(r, {\mathbb C})
$$
be the matrix defined by $R(e_i)\,=\, e_{i+1}$ for
all $i\,\in\, [1\, ,r-1]$ and $R(e_r)\, =\, e_1$, where
$\{e_j\}_{j=1}^r$ is the standard basis of
${\mathbb C}^r$. Let
\begin{equation}\label{rl4}
\rho'_2\, :\, {\mathbb Z}/r{\mathbb Z}\, \longrightarrow\,
\text{PGL}(r,{\mathbb C})
\end{equation}
be the homomorphism that sends any $\tau$ to the
image of $R^\tau$ in $\text{PGL}(r,{\mathbb C})$. Note that
the image of $R$, in $\text{PGL}(r,{\mathbb C})$, commutes
with the image of the above defined diagonal matrix $D$.
Consequently, the images of the two homomorphisms
$\rho'_1$ and $\rho'_2$ commute.

Let
\begin{equation}\label{rl5}
\Psi\, :\, H_1(X,\, {\mathbb Z})\, \longrightarrow\,
\text{PGL}(r,{\mathbb C})
\end{equation}
be the homomorphism defined by
$\gamma\, \longmapsto\, \rho'_1(\rho_J(\gamma))
\rho'_2(\rho_L(\gamma))$, where $\rho_J$ (respectively, $\rho_L$)
is defined in Eq. \eqref{rl1} (respectively, Eq. \eqref{rl2}), and
$\rho'_1$ (respectively, $\rho'_2$)
is defined in Eq. \eqref{rl3} (respectively, Eq. \eqref{rl4}).
Since the images of $\rho'_1$ and $\rho'_2$ commute, the map
$\Psi$ is indeed a homomorphism.

The group $H_1(X,\, {\mathbb Z})$ is a quotient of the fundamental
group of $X$, and $\text{PGL}(r,{\mathbb C})$ has the standard
action on ${\mathbb C}{\mathbb P}^{r-1}$.
Hence any homomorphism from $H_1(X,\, {\mathbb Z})$
to $\text{PGL}(r,{\mathbb C})$ defines a flat projective bundle
over $X$ of relative dimension $r-1$. Let
\begin{equation}\label{pjl}
{\mathcal P}_{J,L}\, \longrightarrow\, X
\end{equation}
be the flat projective bundle given by the homomorphism $\Psi$
in Eq. \eqref{rl5}.

Let $\underline{\text{GL}(r,{\mathbb C})}$ (respectively,
$\underline{\text{PGL}(r,{\mathbb C})}$) be the sheaf of locally
constant functions on $X$ with values in $\text{GL}(r,{\mathbb C})$
(respectively, $\text{PGL}(r,{\mathbb C})$). We have the short
exact sequence of sheaves
\begin{equation}\label{sh0}
e\,\longrightarrow\, \underline{{\mathbb Z}/r{\mathbb Z}}
\,\longrightarrow\, \underline{\text{GL}(r,{\mathbb C})}
\,\longrightarrow\,\underline{\text{PGL}(r,{\mathbb C})}
\,\longrightarrow\, e
\end{equation}
on $X$, where $\underline{{\mathbb Z}/r{\mathbb Z}}$
is the sheaf of locally constant functions with values in ${\mathbb
Z}/r{\mathbb Z}$; for notational convenience, we will
denote the sheaf $\underline{{\mathbb Z}/r{\mathbb Z}}$ by
${\mathbb Z}/r{\mathbb Z}$. Let
\begin{equation}\label{e16}
\chi\, :\, H^1(X,\, \underline{\text{PGL}(r,{\mathbb C})})
\,\longrightarrow\, H^2(X,\, {\mathbb Z}/r{\mathbb Z})
\,=\, {\mathbb Z}/r{\mathbb Z}
\end{equation}
be the homomorphism in the exact sequence of cohomologies
associated to the short exact sequence
in Eq. \eqref{sh0}. The projective bundle
${\mathcal P}_{J,L}$ in Eq. \eqref{pjl} defines an element
$$
c({\mathcal P}_{J,L})\, \in\,
H^1(X,\, \underline{\text{PGL}(r,{\mathbb C})})\, .
$$
Let
\begin{equation}\label{e17}
A_{J,L}\, :=\, \chi(c({\mathcal P}_{J,L}))\, \in\,
H^2(X,\, {\mathbb Z}/r{\mathbb Z})
\,=\, {\mathbb Z}/r{\mathbb Z}
\end{equation}
be the cohomology class, where $\chi$ is the homomorphism in
Eq. \eqref{e16}.

A homomorphism $H_1(X,\, {\mathbb Z})\, \longrightarrow\,
{\mathbb Z}/r{\mathbb Z}$ defines a cohomology class
in $H^1(X,\, {\mathbb Z}/r{\mathbb Z})$. Let
\begin{equation}\label{e-5}
\overline{\rho}_J\, , \overline{\rho}_L\, \in\,
H^1(X,\, {\mathbb Z}/r{\mathbb Z})
\end{equation}
be the cohomology classes corresponding to the homomorphisms
$\rho_J$ and $\rho_L$ constructed in Eq. \eqref{rl1} and
Eq. \eqref{rl2} respectively. Let
$$
\overline{\rho}_J\cup \overline{\rho}_L\,\in\,
H^2(X,\, {\mathbb Z}/r{\mathbb Z})
$$
be the cup product. It can be checked that
\begin{equation}\label{e18}
A_{J,L}\, =\, \overline{\rho}_J\cup \overline{\rho}_L
\, \in\, H^2(X,\, {\mathbb Z}/r{\mathbb Z})\, ,
\end{equation}
where $A_{J,L}$ is constructed in Eq. \eqref{e17}.

Given any holomorphic projective bundle $\mathbb P$ over $X$,
there is a holomorphic vector bundle ${\mathcal V}\,
\longrightarrow\, X$ such that $\mathbb P$ is
isomorphic to the projective bundle over $X$ parametrizing
the lines in the fibers of $\mathcal V$. Let
\begin{equation}\label{e19}
{\mathcal W}\, \longrightarrow\, X
\end{equation}
be a holomorphic vector bundle of rank $r$ such that the
holomorphic
projective bundle over $X$ parametrizing lines in the fibers of
${\mathcal W}$ is holomorphically isomorphic to the
projective bundle ${\mathcal P}_{J,L}$ in Eq. \eqref{pjl}.

\begin{proposition}\label{prop2}
The vector bundle ${\mathcal W}$ in Eq. \eqref{e19} is stable.

The image of ${\rm degree}({\mathcal W})\, \in\, {\mathbb Z}$
in ${\mathbb Z}/r{\mathbb Z}$ coincides with
$A_{J,L}$ in Eq. \eqref{e17}.

Also,
\begin{equation}\label{e20}
End({\mathcal W})\,=\,\bigoplus_{i,j=0}^{r-1} J^{\otimes i}
\bigotimes L^{\otimes j}\, .
\end{equation}
In particular, ${\mathcal W}$ is isomorphic to both
${\mathcal W}\bigotimes J$ and ${\mathcal W}\bigotimes L$.
\end{proposition}

\begin{proof}
Consider the homomorphism $\Psi$ in Eq. \eqref{rl5}.
Its image is a finite subgroup of $\text{PGL}(r,{\mathbb C})$,
and hence $\Psi(H_1(X,\, {\mathbb Z}))$ lies inside
a maximal compact subgroup of $\text{PGL}(r,{\mathbb C})$.
Also, the subgroup
$$
\Psi(H_1(X,\, {\mathbb Z}))\, \subset\, \text{PGL}(r,{\mathbb C})
$$
is irreducible in following sense. Consider the standard action
of $\text{PGL}(r,{\mathbb C})$ on the projective space ${\mathbb
C}{\mathbb P}^{r-1}$ that parametrizes all lines in
${\mathbb C}^r$. The action of the subgroup
$\Psi(H_1(X,\, {\mathbb Z}))$ leaves invariant no proper linear
subspace of ${\mathbb C}{\mathbb P}^{r-1}$.

Since $\Psi(H_1(X,\, {\mathbb Z}))$ is an irreducible subgroup
of $\text{PGL}(r, {\mathbb C})$
lying inside a maximal compact subgroup, it follows that the
principal $\text{PGL}(r, {\mathbb C})$--bundle over $X$
defined by the projective bundle ${\mathcal P}_{J,L}$ (see
Eq. \eqref{pjl}) is stable \cite[p. 146, Theorem 7.1]{Ra}.
Consequently, the corresponding vector bundle ${\mathcal W}$ in
Eq. \eqref{e19} is stable.

{}From the definition of $A_{J,L}$ in Eq. \eqref{e17} it
follows that
$$
{\rm degree}({\mathcal W})\, \equiv\, A_{J,L}\,
\text{~mod~}\, r\, .
$$

To construct the isomorphism in Eq. \eqref{e20}, consider
the homomorphism
$$
h\, :\, {\mathbb Z}/r{\mathbb Z}\,\longrightarrow\,
{\mathbb C}^*
$$
defined by $n\, \longmapsto\, \exp(2\pi\sqrt{-1}n/r)$.
We note that $h\circ \rho_J$ is a character of
$H_1(X,\, {\mathbb Z})$, where $\rho_J$ is the
homomorphism in Eq. \eqref{rl1}. Any character of
$H_1(X,\, {\mathbb Z})$ defines a flat complex line
bundle over $X$ (since $H_1(X,\, {\mathbb Z})$ is
a quotient the fundamental group of $X$, a
character of $H_1(X,\, {\mathbb Z})$ is also a
character of the fundamental group, and hence any
character of $H_1(X,\, {\mathbb Z})$ defines
a flat complex line bundle over $X$). The holomorphic
line bundle corresponding to the character $h\circ
\rho_J$ is $J$ itself. Similarly, the holomorphic line
bundle over $X$ corresponding to the character $h\circ
\rho_L$ of $H_1(X,\, {\mathbb Z})$, where $\rho_L$ is
the homomorphism in Eq. \eqref{rl2}, is identified
with $L$.

Let ${\mathfrak m}(J)$ (respectively,
${\mathfrak m}(L)$) be the one--dimensional complex
$H_1(X,\, {\mathbb Z})$--module defined by the character
$h\circ \rho_J$ (respectively, $h\circ \rho_L$) of
$H_1(X,\, {\mathbb Z})$. The holomorphic line
bundle over $X$ associated to the $H_1(X,\, {\mathbb
Z})$--module ${\mathfrak m}(J)$ (respectively,
${\mathfrak m}(L)$) coincides with the holomorphic line
bundle corresponding to the character $h\circ
\rho_J$ (respectively, $h\circ
\rho_L$) of $H_1(X,\, {\mathbb Z})$. Therefore,
the holomorphic line
bundle over $X$ associated to the $H_1(X,\, {\mathbb
Z})$--module ${\mathfrak m}(J)$ (respectively,
${\mathfrak m}(L)$) coincides with $J$ (respectively, $L$).

On the other hand, consider
the adjoint action of $\text{PGL}(r, {\mathbb C})$
on the vector space $\text{M}(r,{\mathbb C})$ of
$r\times r$--matrices with entries in
$\mathbb C$. Using this action, and the homomorphism
$\Psi$ constructed in Eq. \eqref{rl5}, the vector
space $\text{M}(r,{\mathbb C})$ becomes a
$H_1(X,\, {\mathbb Z})$--module. This
$H_1(X,\, {\mathbb Z})$--module $\text{M}(r,
{\mathbb C})$ has the following decomposition:
\begin{equation}\label{e21}
\text{M}(r,{\mathbb C})\,=\, \bigoplus_{i=0}^{r-1}
\bigoplus_{j=0}^{r-1} {\mathfrak m}(J)^{\otimes i}
\bigotimes {\mathfrak m}(L)^{\otimes j}\,
\end{equation}
where ${\mathfrak m}(J)$ and ${\mathfrak m}(L)$ are
the one--dimensional $H_1(X,\, {\mathbb Z})$--modules
defined above.

The holomorphic vector bundle over $X$ associated to the
above mentioned
$H_1(X,\, {\mathbb Z})$--module $\text{M}(r,{\mathbb C})$
is identified with the vector bundle $End({\mathcal W})$,
where $\mathcal W$ is the vector bundle in Eq. \eqref{e19}.
On the other hand, we noted earlier that the
holomorphic line bundle over $X$
associated to the $H_1(X,\, {\mathbb Z})$--module
${\mathfrak m}(J)$ (respectively, ${\mathfrak m}(L)$)
is $J$ (respectively, $L$). Therefore, fixing an isomorphism
as in Eq. \eqref{e21} we obtain an isomorphism as in
Eq. \eqref{e20}.

A nonzero holomorphic homomorphism
\begin{equation}\label{p}
p\, :\, End({\mathcal W})\, \longrightarrow\, J
\end{equation}
gives a nonzero holomorphic section of
$End({\mathcal W})^\vee\bigotimes J\,=\, End({\mathcal W})
\bigotimes J$. Hence $p$ gives a nonzero holomorphic
homomorphism of vector bundles
$$
{\mathcal W}\, \longrightarrow\,{\mathcal W}\bigotimes J\, .
$$
Since the vector bundle $\mathcal W$ is stable,
and $\text{degree}(J)\,=\, 0$, any nonzero
homomorphism ${\mathcal W}\, \longrightarrow\,{\mathcal W}
\bigotimes J$ must be an isomorphism.

The decomposition in Eq. \eqref{e20} ensures that a
nonzero homomorphism $p$ as in Eq. \eqref{p} exists.
Hence we conclude that $\mathcal W$
is isomorphic to ${\mathcal W}\bigotimes J$. Similarly,
the vector bundle ${\mathcal W}\bigotimes L$
is isomorphic to ${\mathcal W}$. This completes the
proof of the proposition.
\end{proof}

Recall that ${\mathcal M}_\xi (r)^L$ is the fixed point
set defined in Eq. \eqref{e03}.

\begin{lemma}\label{lem2}
Take two linearly independent elements
$J\, , L\, \in\, \Gamma$. Let
$$
\overline{\rho}_J\, , \overline{\rho}_L\,
\in\, H^1(X,\, {\mathbb Z}/r{\mathbb Z})
$$
be the corresponding cohomology classes constructed
as in Eq. \eqref{e-5}. Then
$$
{\mathcal M}_\xi(r)^J\bigcap {\mathcal M}_\xi
(r)^L\,=\, \emptyset
$$
if and only the cup product $\overline{\rho}_J\cup
\overline{\rho}_L\,\in\, H^2(X,\, {\mathbb Z}/r{\mathbb Z})$
vanishes.

If ${\mathcal M}_\xi(r)^J\bigcap {\mathcal M}_\xi
(r)^L\,\not=\, \emptyset$, then
\begin{equation}\label{ca.}
\# ({\mathcal M}_\xi(r)^J\bigcap {\mathcal M}_\xi
(r)^L)\,=\, r^{2g-2}\, .
\end{equation}
\end{lemma}

\begin{proof}
First assume that
\begin{equation}\label{cup}
\overline{\rho}_J\cup
\overline{\rho}_L\, \equiv\, d\, :=\, \text{degree}(\xi)
~\, \text{~mod~}~\, r\, .
\end{equation}
Then from the second part of Proposition \ref{prop2} we know that
$$
\text{degree}({\mathcal W})\, =\, ar +d
$$
for some integer $a$, where $\mathcal W$ is the vector bundle
in Eq. \eqref{e19}. Since $\mathcal W$
is also stable (see the first part of Proposition \ref{prop2}),
there is a holomorphic line bundle
$$
{\mathcal L}_0\, \longrightarrow\, X
$$
of degree $-a$ such that
$$
{\mathcal W}_0\, :=\, {\mathcal W}\bigotimes {\mathcal L}_0\,
\in\, {\mathcal M}_\xi(r)\, .
$$

The vector bundle ${\mathcal W}\, \longrightarrow\, X$ is
isomorphic to both ${\mathcal W}\bigotimes J$
and ${\mathcal W}\bigotimes L$ (see Proposition
\ref{prop2}). Hence ${\mathcal W}_0$ is isomorphic to both
${\mathcal W}_0\bigotimes J$ and ${\mathcal W}_0\bigotimes L$.
Indeed, for any isomorphism
$$
f_0\, :\,  {\mathcal W}\, \longrightarrow\,{\mathcal W}\bigotimes
J\, ,
$$
the homomorphism
$$
f_0\bigotimes \text{Id}_{{\mathcal L}_0}\, :\,
{\mathcal W}_0\, :=\, {\mathcal W}\bigotimes {\mathcal L}_0\,
\longrightarrow\,{\mathcal W}\bigotimes J\bigotimes {\mathcal L}_0
\,=\,{\mathcal W}\bigotimes {\mathcal L}_0\bigotimes J
\,=\, {\mathcal W}_0\bigotimes J
$$
is an isomorphism; similarly, ${\mathcal W}_0$ is isomorphic to
${\mathcal W}_0\bigotimes L$. In other words, we have
\begin{equation}\label{ne}
{\mathcal W}_0\, \in\, {\mathcal M}_\xi(r)^J\bigcap {\mathcal
M}_\xi(r)^L\, .
\end{equation}

Now assume that
\begin{equation}\label{delta}
\delta\, :=\, \overline{\rho}_J\cup
\overline{\rho}_L\, \not=\, 0\, ,
\end{equation}
where $\overline{\rho}_J$ and $\overline{\rho}_L$ are as
in Eq. \eqref{cup}. Fix a positive integer $n_0$ such that
\begin{equation}\label{delta2}
d\, \equiv\, n_0\delta ~\, \text{~mod~}~\, r\, .
\end{equation}
We note that such an integer $n_0$ exists because $r$ is a
prime number, $d\,\not\equiv\, 0$ mod $r$, and
$\delta\, \not=\, 0$. Replace the line bundle $J$ by
\begin{equation}\label{delta3}
J_0\, =\, J^{\otimes n_0}\, ,
\end{equation}
and keep the line bundle $L$ unchanged. From Eq. \eqref{delta}
and Eq. \eqref{delta2},
\begin{equation}\label{c}
\overline{\rho}_{J_0}\cup
\overline{\rho}_L\, \equiv\, d ~\, \text{~mod~}~\, r\, ,
\end{equation}
where $\overline{\rho}_{J_0}\, \in\, H^1(X,\, {\mathbb Z}/
r{\mathbb Z})$ is the cohomology class constructed as in
Eq. \eqref{e-5} for the line bundle $J_0$ in Eq.
\eqref{delta3}.

We noted above that from Eq. \eqref{c} it follows that
$$
{\mathcal M}_\xi(r)^{J_0}\bigcap {\mathcal
M}_\xi(r)^L\, \not=\, \emptyset
$$
(see Eq. \eqref{ne}). Take any
\begin{equation}\label{i}
V\, \in\, {\mathcal M}_\xi(r)^{J_0}\bigcap {\mathcal
M}_\xi(r)^L\, .
\end{equation}
Since $V$ is isomorphic to $V\bigotimes J_0$, it follows
that $V$ is isomorphic to $V\bigotimes J^{\otimes n}_0$ for all
$n$. Indeed, if
$$
f\, :\, V \, \longrightarrow\, V\bigotimes J_0
$$
is an isomorphism, then the composition
$$
V \, \stackrel{f}{\longrightarrow}\, V\bigotimes J_0
\, \stackrel{f\otimes\text{Id}_{J_0}}{\longrightarrow}\,
V\bigotimes J^{\otimes 2}_0\,
\stackrel{f\otimes\text{Id}_{J^{\otimes 2}_0}}{\longrightarrow}\,
V\bigotimes J^{\otimes 3}_0 \stackrel{f\otimes\text{Id}_{J^{\otimes
3}_0}}{\longrightarrow}\cdots
\stackrel{f\otimes\text{Id}_{J^{\otimes (n-1)}_0}}{\longrightarrow}
\,V\bigotimes J^{\otimes n}_0
$$
is an isomorphism for all $n\, \geq\, 1$.

Take a positive integer $m_0$ such that
$m_0n_0\, \equiv\, 1$ mod $r$, where $n_0$ is the integer
in Eq. \eqref{delta2}. Therefore, the line bundle
$J^{\otimes m_0}_0$ is isomorphic to $J$ (see Eq.
\eqref{delta3}). Since $V$ is isomorphic to
$V\bigotimes J^{\otimes m_0}_0\,=\, V\bigotimes J$, we conclude
that
$$
V\, \in\, {\mathcal M}_\xi(r)^{J}\, .
$$
Hence from Eq. \eqref{i},
$$
V\, \in\, {\mathcal M}_\xi(r)^{J}\bigcap {\mathcal
M}_\xi(r)^L\, .
$$

Therefore, we have proved that
$$
{\mathcal M}_\xi(r)^J\bigcap {\mathcal M}_\xi
(r)^L\,\not=\, \emptyset
$$
if $\overline{\rho}_J\cup \overline{\rho}_L\, \not=\, 0$.

To prove the converse, assume that
\begin{equation}\label{a}
\overline{\rho}_J\cup \overline{\rho}_L\, =\, 0\, .
\end{equation}
If $E\,\longrightarrow\, X$ is a holomorphic
vector bundle such that
$$
End(E)\,=\, \bigoplus_{i,j=0}^{r-1} J^{\otimes i}\bigotimes
L^{\otimes j}\, ,
$$
then from Eq. \eqref{a} it follows that $\text{degree}(E)\,
\equiv\, 0$ mod $r$. Since $d\,:=\,\text{degree}(\xi)$ is
coprime to $r$, this implies that
$$
{\mathcal M}_\xi(r)^J\bigcap {\mathcal M}_\xi
(r)^L\,=\, \emptyset\, .
$$

Therefore,
$$
{\mathcal M}_\xi(r)^J\bigcap {\mathcal M}_\xi
(r)^L\,=\, \emptyset
$$
if and only $\overline{\rho}_J\cup \overline{\rho}_L\,=\,0$.

To prove the last statement in the lemma, assume that
$$
{\mathcal M}_\xi(r)^J\bigcap {\mathcal M}_\xi
(r)^L\,\not=\, \emptyset\, .
$$
Fix a vector bundle $E\,\in\, {\mathcal M}_\xi(r)^J
\bigcap {\mathcal M}_\xi (r)^L$. Then from Corollary
\ref{cor1} and Remark \ref{rem1} we know that
${\mathcal M}_\xi(r)^J \bigcap {\mathcal M}_\xi (r)^L$ is
the orbit of $E$ under the action of $\Gamma$ on
${\mathcal M}_\xi(r)$. The isotropy subgroup $\Gamma_E\,
\subset\, \Gamma$ for $E$ is generated by $J$ and
$L$. Now, Eq. \eqref{ca.} holds because
$\#\Gamma\, =\, r^{2g}$, and the order of the subgroup
of $\Gamma$ generated by $J$ and $L$ is $r^2$.
This completes the proof of the lemma.
\end{proof}

\section{The cohomologies}

The cohomology groups $H^i({\mathcal M}_\xi(r),\, {\mathbb Q})$,
$i\, \geq\, 0$, are computed in \cite{HN}, \cite{AB}. The
cohomology algebra $\bigoplus_{i\geq 0}H^i({\mathcal M}_\xi(r),\,
{\mathbb Q})$ is computed by Kirwan in \cite{Ki}.

Consider the
action of $\Gamma$ on $\bigoplus_{i\geq 0}H^i({\mathcal M}_\xi(r),
\, {\mathbb Q})$ given by the action of $\Gamma$ on
${\mathcal M}_\xi(r)$. It is known that this action
is trivial \cite[p. 220, Theorem 1]{HN}. Therefore, the cohomology
algebra $\bigoplus_{i\geq 0}H^i({\mathcal M}_\xi(r),
\, {\mathbb Q})$ is identified with the cohomology
algebra $\bigoplus_{i\geq 0}H^i({\mathcal M}_\xi(r)/\Gamma,
\, {\mathbb Q})$.

Take any nontrivial line bundle $L\, \in\, \Gamma\setminus
\{{\mathcal O}_X\}$. Since $r$ is a prime number, the
order of $L$ is $r$. Let
$$
\gamma_L\, :\, Y_L\, \longrightarrow\, X
$$
be the Galois covering of degree $r$ constructed as in Eq.
\eqref{e5}. Let
\begin{equation}\label{prym}
\text{Prym}_\xi(\gamma_L)\, \subset\, \text{Pic}^d(Y_L)
\end{equation}
be the Prym variety parametrizing all line bundles $\eta\,
\longrightarrow\, Y_L$
such that
$$
\det \gamma_{L*}\eta\, :=\, \bigwedge\nolimits^r
(\gamma_{L*}\eta)\, =\, \xi\, .
$$

Note that the Galois group $\text{Gal}(\gamma_L)$
acts on $\text{Prym}_\xi(\gamma_L)$. The action of
$\tau\, \in\, \text{Gal}(\gamma_L)$ on
$\text{Prym}_\xi(\gamma_L)$ sends any line bundle $\eta$
to $\tau^*\eta$. From Lemma \ref{lem1} we know that
\begin{equation}\label{i2}
\text{Prym}_\xi(\gamma_L)/\text{Gal}(\gamma_L)\,
=\, {\mathcal M}_\xi (r)^L\, .
\end{equation}
We also know that the action of $\text{Gal}(\gamma_L)$
on $\text{Prym}_\xi(\gamma_L)$ is free (see Lemma
\ref{lem-3}).

The group $\Gamma$ acts on $\text{Prym}_\xi(\gamma_L)$. The
action of any $\zeta\, \in\, \Gamma$ is given by the map
$\eta\, \longmapsto\, \eta\bigotimes\gamma^*_L\zeta$; note
that by the projection formula,
$$
\bigwedge\nolimits^r \gamma_{L*}(\eta\bigotimes
\gamma^*_L\zeta)\,=\, (\bigwedge\nolimits^r \gamma_{L*}\eta)
\bigotimes \zeta^{\otimes r}\, =\,
\bigwedge\nolimits^r \gamma_{L*}\eta\, ,
$$
hence $\eta\bigotimes \gamma^*_L\zeta\, \in\,
\text{Prym}_\xi(\gamma_L)$ if $\eta\, \in\,
\text{Prym}_\xi(\gamma_L)$.

It is straight--forward to
check that the actions of $\Gamma$ and
$\text{Gal}(\gamma_L)$ on $\text{Prym}_\xi(\gamma_L)$
commute.

For any $i\, \geq\, 0$, consider the action
of $\Gamma$ on $H^i(\text{Prym}_\xi(\gamma_L),\,
{\mathbb Q})$ given by the above action of $\Gamma$ on
$\text{Prym}_\xi(\gamma_L)$.
Since the action of $\Gamma$ on $\text{Prym}_\xi(\gamma_L)$
is through translations, it follows immediately that
this action of $\Gamma$ on $H^i(\text{Prym}_\xi(\gamma_L),
\, {\mathbb Q})$ is the trivial one.

We will now recall the topological model of a cyclic covering
of $X$ of degree $r$.

The isomorphism classes of unramified cyclic coverings
$$
Y\, \longrightarrow\, X
$$
of degree $r$ with $Y$ connected are parametrized by the
complement
$$
H^1(X,\, {\mathbb Z}/ r{\mathbb Z})_0\, :=\,
H^1(X,\, {\mathbb Z}/ r{\mathbb Z})\setminus\{0\}\, ,
$$
because the space of all the surjective homomorphisms
$$
\pi_1(X,\, x_0)\, \longrightarrow\,{\mathbb Z}/r{\mathbb Z}
$$
is parametrized by $H^1(X,\, {\mathbb Z}/r{\mathbb Z})_0$. Let
$\text{Diff}^+(X)$ denote the group of all orientation preserving
diffeomorphisms of $X$. This group $\text{Diff}^+(X)$ has a
natural action on $H^1(X,\, {\mathbb Z}/r{\mathbb Z})$. The
action of $\text{Diff}^+(X)$ on $H^1(X,\, {\mathbb Z}/r{\mathbb
Z})_0$ can be shown to be transitive. To prove
this, let $\text{Aut}(H^1(X,\, {\mathbb Z}))$ denote the
group of automorphisms of $H^1(X,\, {\mathbb Z})$
preserving the cup product; so $\text{Aut}(H^1(X,\, {\mathbb Z}))$
can be identified with the symplectic group
$\text{Sp}(2g,{\mathbb Z})$ after fixing
a symplectic basis of $H^1(X,\, {\mathbb Z})$. It is known that
the natural homomorphism
$$
\text{Diff}^+(X)\, \longrightarrow\, \text{Aut}(H^1(X,\,
{\mathbb Z}))
$$
is surjective \cite[p. 114]{HL}. On the other hand, it is
easy to check that the natural action of $\text{Aut}(H^1(X,\,
{\mathbb Z}))$ on $H^1(X,\, {\mathbb Z}/r{\mathbb Z})_0$
is transitive. Hence, we conclude that the action of
$\text{Diff}^+(X)$ on $H^1(X,\, {\mathbb Z}/r{\mathbb
Z})_0$ is transitive.

Therefore, given two unramified cyclic coverings
$$
\pi_1\, :\, Y_1\, \longrightarrow\, X ~\,~\,~\,
\text{~and~} ~\,~\,~\, \pi_2\, :\, Y_2\, \longrightarrow\,
X
$$
of degree $r$ with both $Y_1$ and $Y_2$ connected, there
is a diffeomorphism
$$
\varphi\,:\, X\, \longrightarrow\, X
$$
such that $\varphi$ pulls back the covering $\pi_2$ to
$\pi_1$.

One example of a cyclic covering of degree $r$ is the following:

Let $X_0$ be a compact surface of genus one, and let $X_1$
be a compact surface of genus $g-1$. Take an
orientation preserving free action of
${\mathbb Z}/r{\mathbb Z}$ on $X_0$. Let $X'_0$ be the
complement of $r$ open disks in $X_0$ such that
$X'_0$ is preserved by the action of ${\mathbb Z}/r{\mathbb Z}$
on $X_0$. Let $X'_1$ be the complement of a closed disk
in $X_1$. Now attach $r$ copies of $X'_1$ to $X'_0$ along the
$r$ boundary circles of
$X'_0$. The resulting compact connected surface of genus
$r(g-1) +1$ will be denoted by $Y$. The action of
${\mathbb Z}/r{\mathbb Z}$ on $X'_0$ and the permutation action
of ${\mathbb Z}/r{\mathbb Z}$ on the
$r$ copies of $X'_1$ together
define an action of ${\mathbb Z}/r{\mathbb Z}$ on $Y$. This
action is clearly free, and the quotient is of genus $g$.
Up to diffeomorphisms of $Y/({\mathbb Z}/r{\mathbb Z})$, all
connected unramified cyclic coverings of
$Y/({\mathbb Z}/r{\mathbb Z})$ of degree $r$ coincide with
the covering
\begin{equation}\label{Y}
Y\, \longrightarrow\, Y/({\mathbb Z}/r{\mathbb Z})\, .
\end{equation}
In particular, the topological model of the covering
$\gamma_L$ in Eq. \eqref{e5} is the covering in Eq. \eqref{Y}.

Using this model of $\gamma_L$ it follows that
$\text{Prym}_\xi(\gamma_L)$ (defined in Eq.
\eqref{prym}) is topologically isomorphic to a real torus of
dimension $2(r-1)(g-1)$. The action of the Galois group $\text{Gal}
(\gamma_L)$ on $H^1(\text{Prym}_\xi(\gamma_L),\, {\mathbb C})$ can
also be calculated using the above topological model of the covering
$\gamma_L$.

To calculate the action of $\text{Gal}
(\gamma_L)$ on $H^1(\text{Prym}_\xi(\gamma_L),\, {\mathbb C})$,
consider the group ${\mu}_r$ defined in Eq.
\eqref{e04}, which is identified with $\text{Gal}(\gamma_L)$. Let
$$
\widehat{\mu}_r\, :=\, \text{Hom}({\mu}_r\, , {\mathbb C}^*)
$$
be the group of characters of ${\mu}_r$. It is a cyclic
group of order $r$ generated by the tautological
character of ${\mu}_r$ defined by the inclusion of ${\mu}_r$
in ${\mathbb C}^*$.

Each nontrivial element of $\widehat{\mu}_r$
is an eigen--character of ${\mu}_r$ for the action of
$\text{Gal}(\gamma_L)\,=\, {\mu}_r$ on $H^1(\text{Prym}_\xi
(\gamma_L),\, \mathbb C)$, and furthermore, the multiplicity
of each eigen--character is $2(g-1)$.

To prove the above assertion, consider
the covering in Eq. \eqref{Y}. We noted
earlier that it is the topological model of the covering
$\gamma_L$. Let
\begin{equation}\label{A}
A\, :\, H^1(X_1,\, {\mathbb C})^{\oplus r}\, \longrightarrow\,
H^1(X_1,\, {\mathbb C})
\end{equation}
be the homomorphism defined by $$(c_1\, ,\cdots\, ,c_r)\,
\longmapsto\, \sum_{j=1}^r c_j\, ;$$ the surface $X_1$ is the
one used in the construction of the covering in Eq. \eqref{Y}.
The complex vector space $H^1(\text{Prym}_\xi(\gamma_L),\,
\mathbb C)$ is identified with the kernel of the homomorphism
$$
H^1(\text{Pic}^d(Y_L),\, \mathbb C)
\,=\, H^1(Y_L,\, {\mathbb C})\, \longrightarrow\,
H^1(X,\, {\mathbb C}){\mathbb C})
$$
that sends any $c\, \in\, H^1(Y_L,\, {\mathbb C})$
to the class in
$$
H^1(X,\, {\mathbb C}){\mathbb C})\,=\,
H^1(Y_L,\, \mathbb C)^{\text{Gal}(\gamma_L)}
$$
defined by $\sum_{\tau\in \text{Gal}(\gamma_L)}\tau^*c$.
We have
$$
H^1(\text{Pic}^d(Y_L),\, \mathbb C)\, =\,
H^1(X_1,\, {\mathbb C})^{\oplus r}\bigoplus H^1(X_0,\,
{\mathbb C})
$$
($X_0$ is the surface of genus one
in the construction of the covering in Eq. \eqref{Y}),
and
$$
H^1(\text{Prym}_\xi(\gamma_L),\, \mathbb C)\, =\,
\text{kernel}(A)\, \subset\,
H^1(X_1,\, {\mathbb C})^{\oplus r}
\,\subset\, H^1(\text{Pic}^d(Y_L),\, \mathbb C)\, ,
$$
where $A$ is the homomorphism in Eq. \eqref{A}.
The action of
$$
\exp(2\pi\sqrt{-1}/r)\,\in\, \mu_r\,=\,\text{Gal}(\gamma_L)
$$
on $H^1(\text{Pic}^d(Y_L),\, \mathbb C)\, =\,
H^1(X_1,\, {\mathbb C})^{\oplus r}\bigoplus H^1(X_0,\,
{\mathbb C})$ is given by the
automorphism defined by
$$
(c_1\, ,\cdots\, ,c_r\, ; d)\,
\longmapsto\, (c_2\, ,\cdots\, ,c_r\, , c_1\, ; d)
\, \in\, H^1(X_1,\, {\mathbb C})^{\oplus r}\bigoplus
H^1(X_0,\, {\mathbb C})\, .
$$
Also, $\dim H^1(X_1,\, {\mathbb C})\, =\, 2(g-1)$.
It is now easy to see that each nontrivial character of $\mu_r$
is an eigen--character of multiplicity $2(g-1)$ for the action
of $\text{Gal}(\gamma_L)\,=\, {\mu}_r$ on $H^1(\text{Prym}_\xi
(\gamma_L),\, \mathbb C)$.

The cohomology algebra $\bigoplus_{i\geq 0}H^i(\text{Prym}_\xi
(\gamma_L),\, {\mathbb C})$ is identified with the exterior
algebra $\bigoplus_{i\geq 0} \bigwedge^i H^1(\text{Prym}_\xi
(\gamma_L),\, \mathbb C)$. Therefore, from the above description
of the action of $\text{Gal}(\gamma_L)$ on
$H^1(\text{Prym}_\xi (\gamma_L),\, \mathbb C)$ we obtain a
description of the action of $\text{Gal}(\gamma_L)$ on the
cohomology algebra $\bigoplus_{i\geq 0}H^i(\text{Prym}_\xi
(\gamma_L),\, {\mathbb C})$.

\subsection{The Chen--Ruan cohomology}
The case of $r\, =\,2$ was already considered in \cite{BP}. Here we
will assume that $r\, \geq\, 3$.

The $i$--th Chen--Ruan cohomology group is the degree shifted direct 
sum 
 \begin{equation}\label{crgps}
  H^i_{CR}({\mathcal M}_\xi(r)/\Gamma, \, {\mathbb Q} ) \,=\, 
\bigoplus_{L 
\in \Gamma} H^{i- 2 \iota(L)} ({\mathcal M}_\xi(r)^L /\Gamma, \, {\mathbb Q}).
 \end{equation}

 The degree shifting number $\iota(L)$ is obtained from Proposition 
\ref{prop1}:
 \begin{equation}\label{shift} 
  \iota(L) \,=\, \left\{ \begin{array}{ll} 
             0 & {\rm if} \, L = \mathcal{O}_X \\
            \frac{1}{2} (r^2-r)(g-1) & {\rm otherwise.} \end{array} 
\right.       
\end{equation}

As in \cite{BP}, we will denote $H^{{\ast} + 2\iota(L)}({\mathcal 
M}_\xi(r)^L /\Gamma, \, {\mathbb Q})$
by $A^{\ast}(L)$. Then for $\alpha_1 \,\in\, A^p(L_1)$ and $\alpha_2 
\,\in\, A^q(L_2)$, the Chen--Ruan product
$$
\alpha_1 \cup \alpha_2
\,\in\, A^{p+q}(L_1 \bigotimes L_2)
$$
is defined via the relation
\begin{equation}\label{eTP}
 \langle\alpha_1 \cup \alpha_2\, , \alpha_3\rangle\, =\,
\int_{\mathbf{S}/\Gamma} e_1^*\alpha_1 \bigwedge e_2^*
\alpha_2 \bigwedge
e_3^*\alpha_3 \bigwedge c_{\rm top} \mathcal{F}
\end{equation}
for all $\alpha_3 \in A^*(L_3)$ such that $L_1\bigotimes L_2\bigotimes
L_3 \,=\, \mathcal{O}_X$, where $\langle \, ,\, \rangle$ is the 
nondegenerate bilinear Poincar\'e
pairing  for Chen--Ruan cohomology (see \cite{CR1},
\cite[(6.20)]{BP}),
$$
\mathbf{S}\, :=\, \mathcal{M}_\xi(r)^{L_1}\bigcap
\mathcal{M}_\xi(r)^{L_2}\, ,
$$
and $e_i\, :\, \mathbf{S}/\Gamma\, \longrightarrow\,
\mathcal{M}_\xi(r)^{L_i}/\Gamma$
are the canonical inclusions, and 
$\mathcal F$ is a complex
$\Gamma$--bundle over $\mathbf S$,
or equivalently, an orbifold vector bundle over
$\mathbf{S}/\Gamma$, of rank
\begin{equation}\label{erank}
{\rm rank} (\mathcal{F}) \,=\,
\dim_{\mathbb C} \mathbf{S} -
\dim_{\mathbb C} \mathcal{M}_\xi(r) + \sum_{j=1}^3 \iota(L_j)\, .
\end{equation}
 
{}From Lemma \ref{lem2} it follows that $\mathbf{S}$ is empty or zero 
dimensional if
$L_1$ and $L_2$ are linearly independent. If $\mathbf{S}$ is empty, then
the corresponding Chen--Ruan products are automatically zero.
Even otherwise, since $L_3$ is also
nontrivial if $L_1$ and $L_2$ are linearly independent,
using Eq. \eqref{shift} we get that
$${\rm rank}(\mathcal{F})\,=\, \frac{1}{2} (r-1)(r-2)(g-1)
\,>\, 0 \,=\, \dim_{\mathbb C} \mathbf{S}
$$
(recall that $r\, \geq\, 3$).
Hence, $c_{\rm top} \mathcal{F} \,=\, 0$, and again all the 
corresponding Chen--Ruan products are zero.
 
If $L_1$ and $L_2$ are linearly dependent then we have the following 
three cases:

\begin{enumerate}
\item If $L_1$ and $L_2$ are both trivial then the Chen--Ruan products 
are just the usual
products in the singular cohomology of $\mathcal{M}_\xi(r)/\Gamma$.

\item If all the three line bundles are nontrivial, meaning $L_i = 
L^{\otimes k_i}$ for some $L \,\neq\, \mathcal{O}_X$ and
$1\,\le\, k_i \,\le\, (r-1)$, $i\, \in\, \{1\, ,2\, ,3\}$, then
we have $\mathbf{S}\,=\, \mathcal{M}_\xi(r)^L$, and 
$$
{\rm rank}(\mathcal{F})\,=\, \frac{1}{2} (r^2-r)(g-1) \,>\, (r-1)(g-1) 
\,=\, \dim_{\mathbb C} \mathbf{S}
$$
(recall that $r\, \geq\, 3$).  Hence $c_{\rm top} \mathcal{F}
\,=\, 0$, and  all the corresponding Chen--Ruan products are zero. 
 
\item If exactly two of the three $L_i$'s are nontrivial, then we must have 
$L_{i_1}= L$,  $L_{i_2} = L^{\otimes (r-1)}$ and $ L_{i_3}= \mathcal{O}_X$ 
for some $L \neq \mathcal{O}_X$ and
some permutation $\{i_1,i_2, i_3 \}$ of $\{1,2,3 \}$. The calculations 
of the Chen--Ruan
products are analogous to the cases a), b) and c) in Section 6.1 of \cite{BP}.
 
\end{enumerate} 

\section*{Acknowledgments}
The first author wishes to thank the Indian
Statistical Institute for hospitality while the
work was carried out.



\begin{thebibliography}{AAAA}

\bibitem{AR} A. Adem and Y. Ruan, Twisted orbifold $K$--theory,
\textit{Comm. Math. Phys.} \textbf{237} (2003) 533--556.

\bibitem{AB} M. F. Atiyah and R. Bott, The Yang--Mills
equations over Riemann surfaces, \textit{Phil. Trans. Roy. Soc.
Lond.} \textbf{308} (1982) 523--615.

\bibitem{BP} I. Biswas and M. Poddar, Chen--Ruan cohomology of
some moduli spaces, \textit{Int. Math. Res. Not.}
(2008), article ID rnn104, 32 pages, doi:10.1093/imrn/rnn104.

\bibitem{CR1} W. Chen and Y. Ruan, A new cohomology theory of
orbifold, \textit{Comm. Math. Phys.} \textbf{248} (2004) 1--31.

\bibitem{CR2} W. Chen and Y. Ruan, Orbifold Gromov--Witten
theory, in: \textit{Orbifolds in mathematics and physics}
(Madison, WI, 2001), pp. 25--85, Contemp.
Math. \textbf{310}, Amer. Math. Soc., Providence, RI, 2002.

\bibitem{HL} R. Hain and E. Looijenga, Mapping class
groups and moduli spaces of curves, in:
\textit{Algebraic geometry---Santa Cruz 1995}, 97--142,
Proc. Sympos. Pure Math., 62, Part 2, Amer. Math. Soc.,
Providence, RI, 1997.

\bibitem{HN} G. Harder and M. S. Narasimhan, On the
cohomology groups of moduli spaces of vector bundles on
curves, \textit{Math. Ann.} \textbf{212} (1975) 215--248.

\bibitem{Ki} F. Kirwan, The cohomology rings of moduli
spaces of bundles over Riemann surfaces, \textit{Jour. Amer.
Math. Soc.} \textbf{5} (1992) 853--906.

\bibitem{Mu} D. Mumford, Prym varieties. I, in:
\textit{Contributions to analysis (a collection of papers
dedicated to Lipman Bers)}, pp. 325--350,
Academic Press, New York, 1974.

\bibitem{Ra} A. Ramanathan, Stable principal bundles on a
compact Riemann surface, Math. Ann. \textbf{213}
(1975), 129--152.

\bibitem{Ru} Y. Ruan, Stringy geometry and topology of
orbifolds, in: \textit{Symposium in Honor of C. H. Clemens}
(Salt Lake City, UT, 2000), pp. 187--233, Contemp.
Math. \textbf{312}, Amer. Math. Soc., Providence, RI, 2002.

\end{thebibliography}
\end{document}